 \documentclass[final,nomarks]{dmtcs-episciences}
\usepackage{preamble}
\tikzset{edgee/.style = {->,> = latex'}}
\newcolumntype{P}[1]{>{\centering\arraybackslash}p{#1}}
\newcolumntype{M}[1]{>{\centering\arraybackslash}m{#1}}

\title{A Study On The Graph Formulation Of Union Closed Sets Conjecture}

\author{Nived J M}
\affiliation{Indian Institute of Technology Hyderabad, Hyderabad, India}

\keywords{Union closed sets conjecture, maximal stable sets, graph decompositions, bipartite graphs, pendant vertices}

\begin{document}
\publicationdata{vol. 28:3}{2026}{11}{10.46298/dmtcs.14730}{2024-11-13; 2024-11-13; 2025-10-03; 2026-07-24}{2026-08-19}
\maketitle

\begin{abstract}
 \vspace{0.3cm}
The Union Closed Sets Conjecture is one of the most renowned problems in combinatorics. Its appeal lies in the simplicity of its statement contrasted with the potential complexity of its resolution. The conjecture posits that, in any union closed family of sets, there exists at least one element that appears in at least half of the sets within the family. We establish the graph-theoretic version of certain set-theoretic results by connecting the set-based and graph-based formulations. We then prove a theorem in which we investigate the conjecture for graphs, focusing on their decompositions and the position of certain pendant vertices. As a result, we extend the validity of the conjecture to a broader class of graph structures.
\end{abstract}

\section{Introduction}\label{sec1}

The Union Closed Sets Conjecture (UCC), also known as Frankl's Conjecture, was introduced by Peter Frankl in 1979. This conjecture pertains to families of sets where the union of any two member sets remains a member of the family, called \emph{union closed} families. For a given union closed family $\mathcal{F}$, the \emph{universe} $U(\mathcal{F})$ is the union of all its member sets, and the \emph{frequency} of an element $x \in U(\mathcal{F})$ is the count of member sets containing $x$. The conjecture asserts that for every finite union closed family $\mathcal{F}$ ($\neq \{\emptyset\}$), there exists an \emph{abundant} element, defined as an element whose frequency is at least $\frac{|\mathcal{F}|}{2}$. A family is said to satisfy the \emph{UCC condition} if it contains an abundant element.

 Despite over four decades of study, including the recent work of \cite{gil} establishing a constant lower bound using entropy methods, the conjecture remains unsolved in its general form. This shows the importance of verifying the conjecture for specific families. Here, we outline known cases where the UCC holds:

\begin{enumerate}
    \item When $|U(\mathcal{F})| \leq 12$ \cite{12el}. \label{it1} 
  \item When $|\mathcal{F}| \leq 50$ \cite{50c,50b,12el}.  \label{it2}
    \item When $|\mathcal{F}| \leq 2|U(\mathcal{F})|$ and \( A \) is separating, meaning that any two distinct elements have at least one subset containing only one of them \cite{2m}.    \label{it3}
    \item When $|\mathcal{F}| \geq \frac{2}{3}2^{|U(\mathcal{F})|}$  \cite{twot}.    \label{it4}
    \item When the family includes a set of cardinality 1 or 2 \cite{sin,50}.   \label{it5}
    \item When each member set has a cardinality of at least $\frac{|U(\mathcal{F})|}{2}$ \cite{sin}.     \label{it6}
\end{enumerate}

\cite{11} introduced a graph-theoretic approach to the Union Closed Sets Conjecture (UCC). A subset of vertices in a graph is termed \emph{stable} if no two vertices in it are adjacent. A stable set is \emph{maximal} if no additional vertices can be added without losing this property.

\begin{conjecture}[Graph formulation of UCC]\label{cond}
For any bipartite graph with at least one edge, there exists a vertex in each bipartite class that lies in at most half of the maximal stable sets.
\end{conjecture}

A vertex as described in Conjecture~\ref{cond} is termed \emph{rare}. We say that a bipartite graph satisfies the \emph{UCC condition} if it contains a rare vertex in each of its bipartite classes. This formulation is the dual of the set version of UCC, replacing an element appearing in many sets with a vertex appearing in few maximal stable sets.

In Section \ref{sec3}, we establish a connection between the set-based and graph-based formulations of the conjecture. Using this connection, we present the graph version of the known results as a natural consequence. In Section \ref{sec4}, we study the enumeration of maximal stable sets through graph decompositions, leading to a general theorem that provides new insights into the UCC and establishes its validity for several additional graph classes. Finally, in Section \ref{secf}, we outline potential directions for future research and discuss open problems motivated by our findings.

\section{Connecting The Two Formulations}\label{sec3}

The union closed family generated by a family \( \mathcal{F} \), denoted \( \langle \mathcal{F} \rangle \), consists of the collection of all unions of its member sets, including the empty set \( \emptyset \). For a graph, the set of neighbors of a vertex subset \( P \) is denoted by \( N(P) \). If \( P = \{x\} \) is a singleton, we use \( N(x) \) for convenience\footnote{When multiple graphs are involved, we use \( N_{\scriptscriptstyle G}(P) \), and \( N_{\scriptscriptstyle G}(x) \) respectively, to specify the graph.}. A \emph{pendant vertex} is a vertex of degree 1. For a graph \( G \), we denote the set of all maximal stable sets by \(\mathcal{B}_{\scriptscriptstyle G}\). The cardinality of \(\mathcal{B}_{\scriptscriptstyle G}\) is denoted by \( w_{\scriptscriptstyle G} \), which represents the number of maximal stable sets of \( G \). For any vertex \( x \in V(G) \), the number of maximal stable sets containing \( x \) is represented by \( w_{\scriptscriptstyle G}(x) \). 

Throughout this section, we will refer to the two vertex classes of the bipartite graph \( G \) as \( X \) and \( Y \). We define the \emph{incidence graph} of \( \mathcal{F} \) as the bipartite graph with bipartition \( X = \mathcal{F} \) and \( Y = U(\mathcal{F}) \), and edge set \( \{ (S, x) : S \in \mathcal{F},\, x \in U(\mathcal{F}),\, x \in S \} \). For a bipartite graph \( G(X,Y) \), we refer to the \emph{incidence family} of \( X \) as \( \mathcal{F}^{\scriptscriptstyle X} = \{ N(y) : y \in Y \} \). If there are no isolated vertices\footnote{We can omit isolated vertices without affecting the conjecture, since an isolated vertex lies in every maximal stable set and is therefore abundant. Its presence is redundant: removing it does not change the number of maximal stable sets or the frequency of any other vertices.} in \( X \), then its universe is given by \( U(\mathcal{F}^{\scriptscriptstyle X}) = X \). As an example, Figure~\ref{fig:p1} illustrates the incidence graph corresponding to the family  \(
\mathcal{F}^{\scriptscriptstyle X} = \big\{ \{1,2,3\}, \{2\}, \{2,3\}, \{3,4\}, \{4,5\} \big\}.
\)  
In this graph, the vertices \( 2, 3, \) and \( 4 \) are rare, and correspondingly, the elements \( 2, 3, \) and \( 4 \) are abundant in the union closed family generated by \( \mathcal{F}^{\scriptscriptstyle X} \).

The following proposition gives a direct connection between the set and graph formulations of the UCC, showing that an abundant element in \(\langle \mathcal{F}^{\scriptscriptstyle X} \rangle\) corresponds exactly to a rare vertex in \(X\). In \cite{11} this is proved using the incidence graph of a union closed family. Here we instead use the incidence graph of a generating family, giving a more general result. The proof is inspired by \cite{11}, with the correspondence stated explicitly.

\begin{proposition}\label{prop31}
    Let $G$ be a bipartite graph with vertex partition $X\cup Y$. A vertex $x\in X$ is rare if and only if it is abundant in $\langle \mathcal{F}^{\scriptscriptstyle X} \rangle$.
\end{proposition}

\begin{proof}
Consider any subset of vertices from \(Y\), denoted as \(Y^{\prime}\). It can be observed that there exists a unique maximal stable set in \(G\), denoted as \(S\), such that \(S\cap X = X\setminus N(Y^{\prime})\). Furthermore, every maximal stable set \(S\) satisfies \(S\cap X = X\setminus N(Y^{\prime})\) for some \(Y^{\prime} \subseteq Y\). Hence \(\{X\setminus N(Y^{\prime}) \mid Y^{\prime} \subseteq Y\} = \{S \cap X \mid S \in \mathcal{B}_{\scriptscriptstyle G}\}\). It is important to note that no two maximal stable sets \(S_{\scriptscriptstyle 1}\) and \(S_{\scriptscriptstyle 2}\) can have \(S_{\scriptscriptstyle 1}\cap X = S_{\scriptscriptstyle 2}\cap X\). Therefore, the set of maximal stable sets is in one-to-one correspondence with \(\{S \cap X \mid S \in \mathcal{B}_{\scriptscriptstyle G}\} = \{X\setminus N(Y^{\prime}) \mid Y^{\prime} \subseteq Y\}\). Note that, a vertex \(x \in X\) is rare in \(G\) if and only if it is rare in the family \(\{X\setminus N(Y^{\prime}) \mid Y^{\prime} \subseteq Y\}\), which holds if and only if \(x\) is abundant in the family \(\{N(Y^{\prime}) \mid Y^{\prime} \subseteq Y\}\). Now, let $F=\{N(y)\hspace{.2 cm}\vert \hspace{.2 cm} y\in Y^{\prime}\}$, and observe that the universe $U(F)$ is nothing but $N(Y^{\prime})$. It is also evident that $\{ N(Y^{\prime}) \hspace{.2 cm}\vert \hspace{.2 cm} Y^{\prime}\subseteq Y\}=\{U(F)\hspace{.2 cm}\vert \hspace{.2 cm}F\subseteq\mathcal{F}^{\scriptscriptstyle X}\}=\langle \mathcal{F}^{\scriptscriptstyle X} \rangle$.
\end{proof}

\begin{figure}[h!]
\centering
\begin{minipage}{.475\textwidth}
  \centering

 \includegraphics[width=0.99\textwidth]{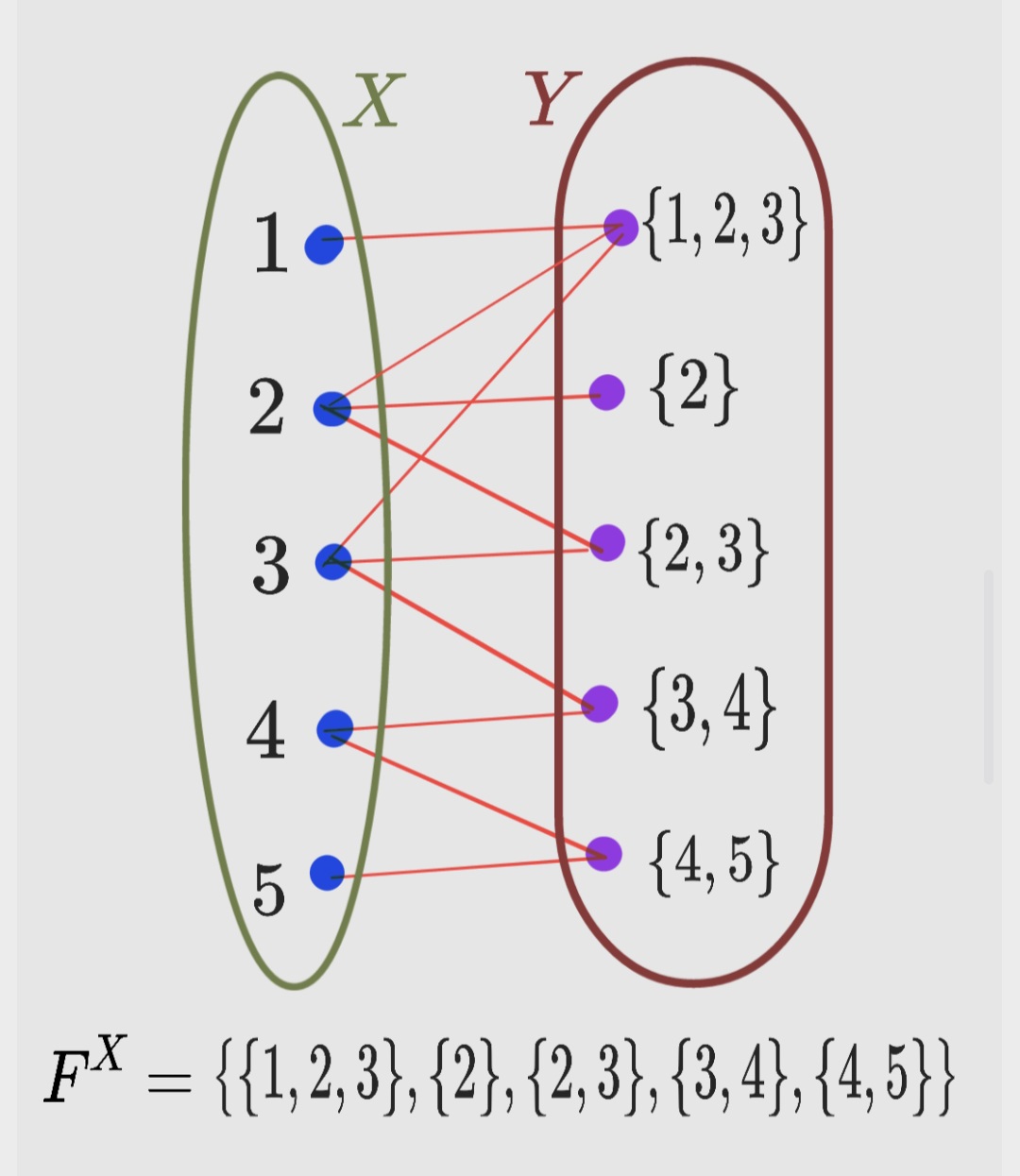}

 \caption{An illustration of an incidence family together with its corresponding incidence graph.}

  \label{fig:p1}

\end{minipage}%
\hfill
\begin{minipage}{.485\textwidth}
  \centering
 
  \includegraphics[width=0.95\textwidth]{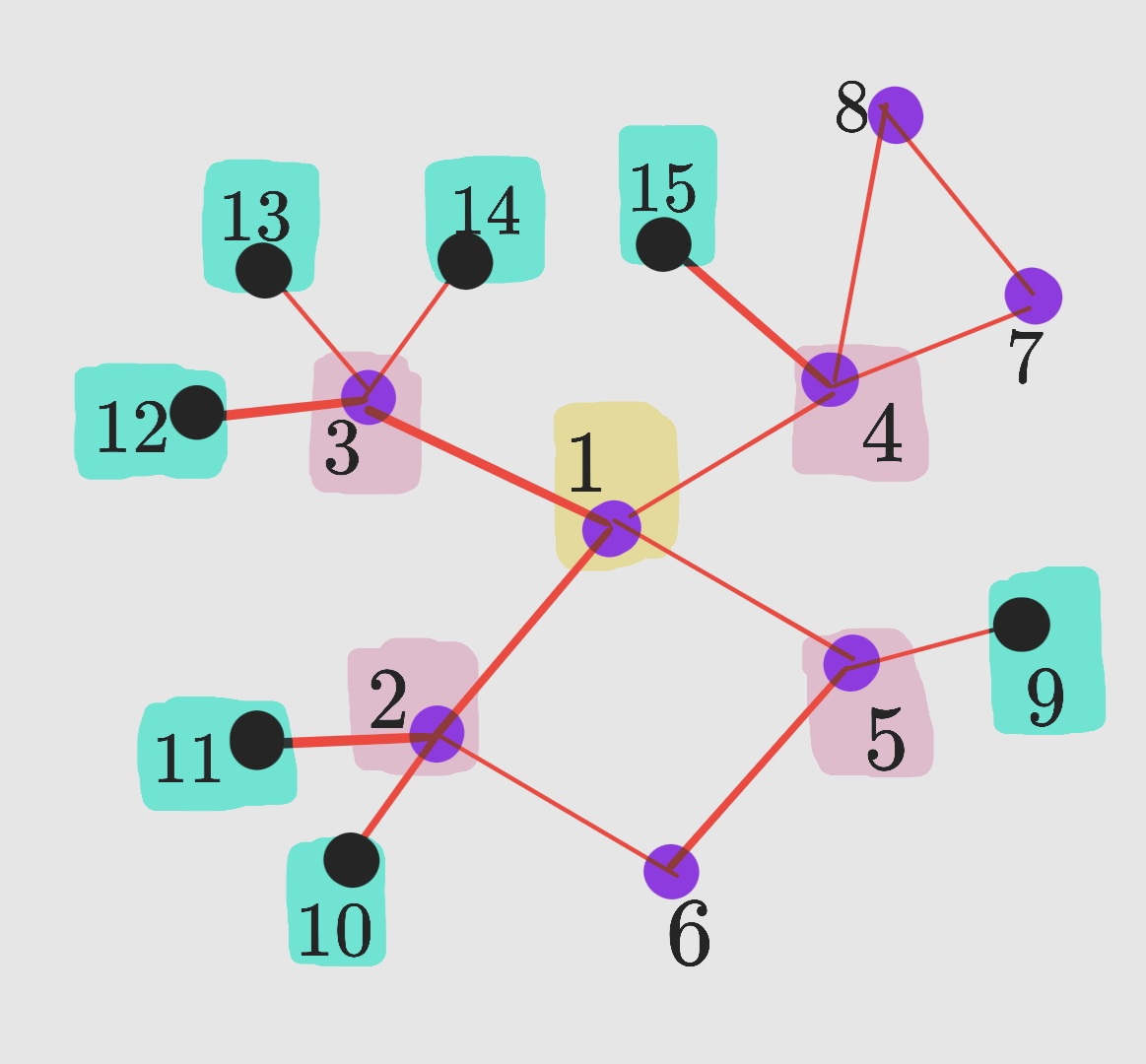}

\caption{Illustration of a 2-layered vertex in a graph, where the vertex labeled 1 is 2-layered.}

  \label{fig:p9}
 
\end{minipage}
\end{figure}

The following are several observations, comprising graph-theoretic adaptations of certain established results.

\begin{itemize}
    \item If $G$ is a bipartite graph with one of its two classes containing at most 12 vertices, then there exists a rare vertex within the same class. (corresponding to Result \ref{it1})
    \item  A bipartite graph with both classes of cardinality at most 12 satisfies UCC condition. (follows immediately from the previous bullet point; corresponds to Result~\ref{it1}.)
    \item In a twin-free\footnote{No two vertices share the same neighborhood.} bipartite graph $G=(X,Y)$, if $2^{\scriptscriptstyle |Y|} \leq 2|X|$, then there exists a rare vertex in $X$.  (corresponding to Result \ref{it3})
    \item In a twin-free bipartite graph, if $|Y|\geq \tfrac{2}{3}2^{\scriptscriptstyle |X|}$, then the graph satisfies the UCC condition. (Existence of a rare vertex in $X$ follows directly from Result \ref{it4}, and the existence of a rare vertex in $Y$ follows directly from the previous bullet point.)
    \item  If a bipartite graph $G$ contains a pendant vertex, then its neighbor is rare in $G$. Additionally, if $G$ contains a vertex of degree 2, then there exists a rare vertex in the other bipartite class.    (Observed in \cite{11}; also follows from the set version, Result \ref{it5})
    \item  Let $\mathcal{F}$ be a family in which every member set has a maximum size of 3, and each element within the sets has a frequency that does not exceed 3. Then $\langle \mathcal{F}\rangle$ satisfies UCC condition.   (follows from the result that subcubic graphs satisfy the UCC condition \cite{11})
    \item Let $G$ be a bipartite graph with bipartite classes $X$ and $Y$. If $deg(a)\geq\frac{|X|}{2}$ for every vertex $a\in Y$, then there is a rare vertex in $X$.  (corresponding to Result \ref{it6})
    \item  Let $G$ be a bipartite graph with a minimum degree of $\delta$. If both bipartite classes have cardinalities at most $2\delta$, then $G$ satisfies UCC condition.  (Follows directly from the previous bullet point; corresponds to Result~\ref{it6})
    \item  An $r$-regular bipartite graph with both of its bipartite classes having cardinality at most $2r$ satisfies UCC condition. (Consequence of the previous 2 bullet points; corresponds to Result \ref{it6})
\end{itemize}

 In the subsequent section, our focus will be on graphs that possess pendant vertices positioned in specific locations.

\section{Some New Graph Results}\label{sec4}

In this section, we focus on enumerating the number of maximal stable sets of a graph by decomposing it into smaller subgraphs. We begin with the simplest case, when the graph is disconnected, which is addressed in the following proposition. We then move to a more general setting, where the graph is decomposed into two subgraphs without common edges but whose union forms the parent graph. By analyzing the properties of their common vertices, we introduce a new approach to compute the maximal stable sets of the parent graph. Later, we present a result derived directly from this theorem, which establishes the UCC for several new classes of graphs.

\begin{proposition}\label{prop1}
The total number of maximal stable sets of a graph G is the product of the numbers of maximal stable sets of its connected components. Moreover, in a disconnected bipartite graph G, a vertex is rare if and only if it is rare in the connected component to which it belongs.
\end{proposition}

\begin{proof}
    Suppose \(P\) is a component of \(G\) and set \(Q = G - P\). Let $S_{\scriptscriptstyle 1}$ and $S_{\scriptscriptstyle 2}$ be two maximal stable sets of $P$ and $Q$ respectively. One can notice that $S_{\scriptscriptstyle 1} \cup S_{\scriptscriptstyle 2}$ is a maximal stable set of $G$. For any maximal stable set $S$ of $G$, observe that $S \cap V(P)$ and $S \cap V(Q)$ are maximal stable sets of $P$ and $Q$ respectively. This gives the number of maximal stable sets of $G$, $w_{\scriptscriptstyle G} = w_{\scriptscriptstyle P} w_{\scriptscriptstyle Q}$. One can see that a maximal stable set $S$ of $G$ contains a vertex $x \in V(P)$ if and only if $S$ is of the form $S_{1} \cup S_{2}$ where $S_{1}$ and $S_{2}$ are maximal stable sets of $P$ and $Q$ respectively and $x \in S_{1}$. Hence, for every $x \in V(P)$, $w_{\scriptscriptstyle G}(x) = w_{\scriptscriptstyle P}(x) w_{\scriptscriptstyle Q}$.

Now, let $G$ be a bipartite graph and $r$ be a vertex in $P$. Then by definition, $r$ is rare in $P$ if and only if $w_{\scriptscriptstyle P} - 2w_{\scriptscriptstyle P}(r) \geq 0$. As a result, $w_{\scriptscriptstyle G} - 2w_{\scriptscriptstyle G}(r) = w_{\scriptscriptstyle P} w_{\scriptscriptstyle Q} - 2w_{\scriptscriptstyle P}(r) w_{\scriptscriptstyle Q} = (w_{\scriptscriptstyle P} - 2w_{\scriptscriptstyle P}(r)) w_{\scriptscriptstyle Q} \geq 0,$ which shows that $r$ is rare in $G$ if and only if $r$ is rare in $P$.
\end{proof}



We now present the main structural result. A \emph{decomposition} of a graph $G$ is a set of edge-disjoint subgraphs $H_{1}, H_{2}, \ldots, H_{n}$ such that $\bigcup_{\scriptscriptstyle i=1}^{\scriptscriptstyle n} H_{i} = G$. A vertex $v \in V(G)$ is defined as \emph{2-layered} if, for every neighbor $u$ of $v$, there exists a pendant vertex adjacent to $u$ that is different from $v$. For instance, in Figure~\ref{fig:p9}, the vertex labeled \( 1 \) is 2-layered. The vertex \( 1 \) has four neighbors \( 2, 3, 4, \) and \( 5 \). Each of these four vertices has at least one pendant neighbor.

\begin{theorem}\label{thm4}
    Let $G$ be a graph and let $\{H, I\}$ be a decomposition of $G$. Suppose that the vertices in $V(H)\cap V(I)$ form a stable set and are all $2$-layered. Then the total number of maximal stable sets of $G$ is the product of the numbers of maximal stable sets of $H$ and $I$.
\end{theorem}

The following consequence is established by the same argument as Proposition \ref{prop1}, using Theorem \ref{thm4}, which proves the UCC condition for some classes of graphs.

\begin{theorem}\label{thm42}
    Let $G$ be a bipartite graph and $\{H,I\}$ a decomposition of $G$. Suppose the vertices in $ V(H) \cap  V(I)$ belong to the same bipartite class of $G$ and are all 2-layered. Then, all rare vertices in $H$ or $I$ will remain rare in $G$.
\end{theorem}

Before presenting the proofs of the theorems, we first establish some notation and lemmas. The following notation will be used throughout the section. We define $[n] = \{i \in \naturals \mid i \leq n\}$, and for any subset $\Theta \subseteq [n]$, we denote its complement by $\Theta^{\mathsf{c}} = [n] \setminus \Theta$. Recall the definitions of the set of all maximal stable sets $\mathcal{B}_G$ and its cardinality $w_G$.
Here, we will slightly generalize these definitions. For any vertex subsets \( P, Q \subseteq V(G) \) and any subgraph \( K \) of \( G \), the notation \( \mathcal{B}_{\scriptscriptstyle K}(P, \overline{Q}) \) denotes the set of all maximal stable sets of \( K \) that include all vertices in \( P \) and exclude all vertices in \( Q \). The cardinality of \( \mathcal{B}_{\scriptscriptstyle K}(P, \overline{Q}) \) is denoted by \( w_{\scriptscriptstyle K}(P, \overline{Q}) \). We write $K \setminus P$ to denote the subgraph obtained by removing the vertices in $P$ and all edges incident to them from $K$.

\begin{figure}[h!]
\centering
\begin{minipage}{.425\textwidth}
  \centering

 \includegraphics[width=0.95\textwidth]{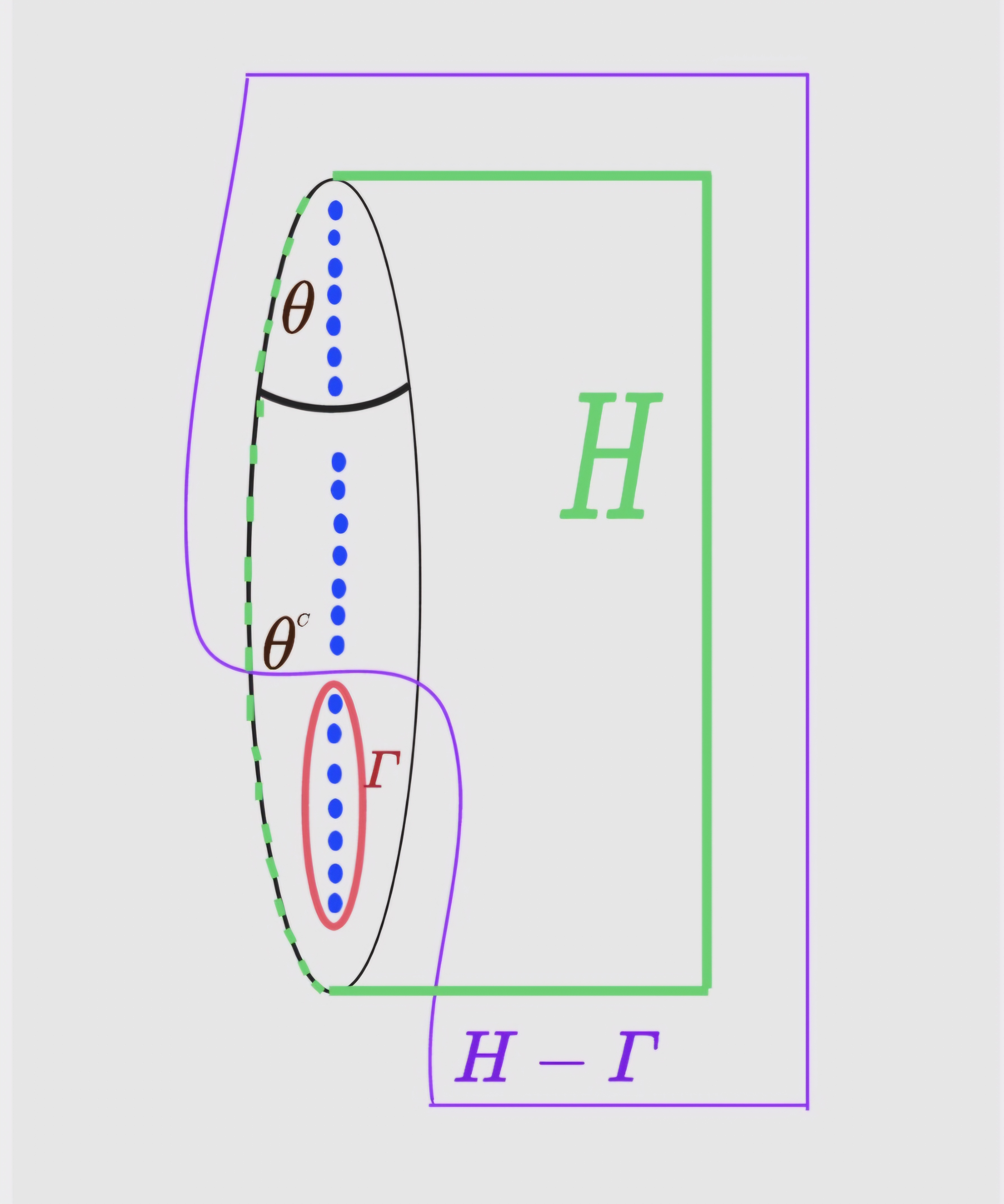}
 
 \caption{Illustration of Lemma \ref{lem4} with graph \(H\), stable set $[n]$ consisting of all 2-layered vertices, and the subset \(\Gamma \subseteq \Theta^{\mathsf{c}} \).}
  \label{fig:p1in}

\end{minipage}%
\hfill
\begin{minipage}{.525\textwidth}
  \centering
 
  \includegraphics[width=0.95\textwidth]{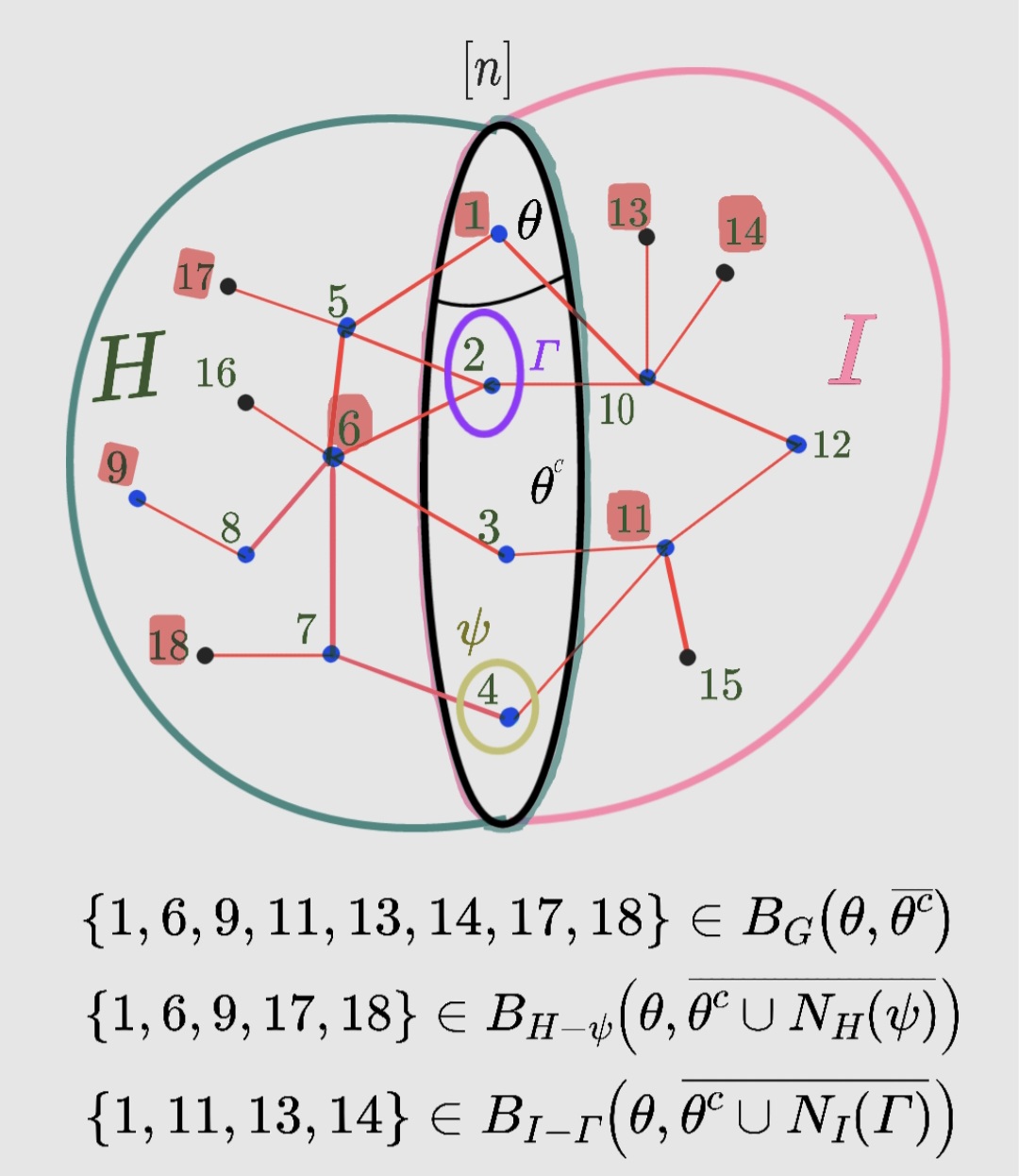}

\caption{An illustration of Lemma \ref{lem7}, showing the decomposition of graph \(G\) into subgraphs \(H\) and \(I\). The common vertex set \([n]\) is highlighted together with its partition into \(\Theta\), \(\Theta^{\mathsf{c}}\), \(\Gamma \subseteq \Theta^{\mathsf{c}}\), and \(\Psi \subseteq \Theta^{\mathsf{c}} \setminus \Gamma\).}

  \label{fig:p2}
 
\end{minipage}
\end{figure}

\begin{lemma}\label{lem4}
  Let $H$ be a graph with $[n] \subseteq V(H)$ such that $[n]$ is a stable set and all vertices of $[n]$ are 2-layered. Let $\Theta \subseteq [n]$, $\Gamma \subseteq \Theta^{\mathsf{c}}$, and $a \in V(H) \setminus \Gamma$. If $\Gamma$ contains no pendant vertices, then the following identities hold:
\begin{equation}\label{eq1}
    w_{\scriptscriptstyle H\setminus\Gamma}(\Theta, \overline{\Theta^{\mathsf{c}}\cup N_{\scriptscriptstyle H}(\Gamma)}) = w_{\scriptscriptstyle H}(\Theta \cup \Gamma, \overline{\Theta^{\mathsf{c}} \setminus \Gamma})
\end{equation}
\begin{equation}\label{eq2}
    w_{\scriptscriptstyle H\setminus\Gamma}(\{a\}\cup \Theta, \overline{\Theta^{\mathsf{c}}\cup N_{\scriptscriptstyle H}(\Gamma)}) = w_{\scriptscriptstyle H}(\{a\}\cup \Theta \cup \Gamma, \overline{\Theta^{\mathsf{c}} \setminus \Gamma})
\end{equation}
\end{lemma}

\begin{proof}
Consider the mapping $f: \mathcal{B}_{\scriptscriptstyle H\setminus\Gamma}(\Theta, \overline{\Theta^{\mathsf{c}}\cup N_{\scriptscriptstyle H}(\Gamma)}) \to \mathcal{B}_{\scriptscriptstyle H}(\Theta \cup \Gamma, \overline{\Theta^{\mathsf{c}}\setminus \Gamma})$ defined by $ f(B) = B \cup \Gamma $. Let $B \in \mathcal{B}{\scriptscriptstyle H\setminus\Gamma}(\Theta, \overline{\Theta^{\mathsf{c}} \cup N{\scriptscriptstyle H}(\Gamma)})$, observe that \( f(B) \) includes all vertices from \( \Theta \cup \Gamma \) and excludes all vertices from \( \Theta^{\mathsf{c}} \setminus \Gamma \). Since \( B \) is a stable set in \( H \setminus \Gamma \) and contains no vertices from \( N_{\scriptscriptstyle H}(\Gamma) \), it follows that \( f(B) \) is stable in \( H \). Assume, for contradiction, that \( f(B) \) is not a maximal stable set in \( H \). Then, there exists a vertex \( v \in V(H) \) such that \( v \notin f(B) \) and \( B \cup \Gamma \cup \{v\} \) is stable in \( H \). This implies that \( B \cup \{v\} \) is stable in \( H \setminus \Gamma \), contradicting the maximality of \( B \) in \( H \setminus \Gamma \). Hence, \( f(B) \in \mathcal{B}_{\scriptscriptstyle H}(\Theta \cup \Gamma, \overline{\Theta^{\mathsf{c}} \setminus \Gamma}) \).

Next, consider a set \( D \in \mathcal{B}_{\scriptscriptstyle H}(\Theta \cup \Gamma, \overline{\Theta^{\mathsf{c}} \setminus \Gamma}) \). We define the inverse image of \( D \) under the mapping \( f \) as \( f^{-1}(D) = D \setminus \Gamma \). Since \( D \) necessarily includes all vertices from \( \Gamma \), it excludes any vertices from \( N_{\scriptscriptstyle H}(\Gamma) \). Therefore, \( f^{-1}(D) \) is clearly a subset of \( V(H \setminus \Gamma) \) that contains all vertices in \( \Theta \), while simultaneously avoiding all vertices in \( \Theta^{\mathsf{c}} \setminus \Gamma \) and \( N_{\scriptscriptstyle H}(\Gamma) \).

Also note that $D$ is a stable set, and since $f^{-1}(D)$ is a subset of $D$, it is also a stable set. Suppose, for contradiction, that $f^{-1}(D)$ is not maximal stable in $H \setminus \Gamma$. This implies there exists a vertex $v \in V(H \setminus \Gamma)$, with $v \notin f^{-1}(D)$, such that the set $\{v\} \cup f^{-1}(D)$ is stable in $H \setminus \Gamma$. Since all vertices in $\Gamma$ are 2-layered and $\Gamma$ contains no pendant vertices, each vertex in $N_{\scriptscriptstyle H}(\Gamma)$ has an adjacent pendant vertex in $H\setminus\Gamma$. Due to the maximality of the stable sets, these pendant vertices must be included in all member sets of $\mathcal{B}_{\scriptscriptstyle H}(\Theta \cup \Gamma, \overline{\Theta^{\mathsf{c}} \setminus \Gamma})$. Consequently, they are present in every member set of $f^{-1}(\mathcal{B}_{\scriptscriptstyle H}(\Theta \cup \Gamma, \overline{\Theta^{\mathsf{c}} \setminus \Gamma}))$. Thus, $v \notin N_{\scriptscriptstyle H}(\Gamma)$. Therefore, $\{v\} \cup D = \{v\} \cup f^{-1}(D) \cup \Gamma$ would be stable in $H$, which contradicts the assumption that $D$ is a maximal stable set in $H$. Hence, $f^{-1}(D) \in \mathcal{B}_{\scriptscriptstyle H \setminus \Gamma}(\Theta, \overline{\Theta^{\mathsf{c}} \cup N_{\scriptscriptstyle H}(\Gamma)})$ and the mapping is well-defined.

It is trivial that \( f \) is a bijective map, and this correspondence between the families \( \mathcal{B}_{\scriptscriptstyle H \setminus \Gamma}(\Theta, \overline{\Theta^{\mathsf{c}} \cup N_{\scriptscriptstyle H}(\Gamma)}) \) and \( \mathcal{B}_{\scriptscriptstyle H}(\Theta \cup \Gamma, \overline{\Theta^{\mathsf{c}} \setminus \Gamma}) \) establishes Equation \ref{eq1}. Now note that \( \mathcal{B}_{\scriptscriptstyle H \setminus \Gamma}(\{a\}\cup\Theta, \overline{\Theta^{\mathsf{c}} \cup N_{\scriptscriptstyle H}(\Gamma)}) \) and \( \mathcal{B}_{\scriptscriptstyle H}(\{a\}\cup\Theta \cup \Gamma , \overline{\Theta^{\mathsf{c}} \setminus \Gamma}) \) are precisely the subfamilies of \( \mathcal{B}_{\scriptscriptstyle H \setminus \Gamma}(\Theta, \overline{\Theta^{\mathsf{c}} \cup N_{\scriptscriptstyle H}(\Gamma)}) \) and \( \mathcal{B}_{\scriptscriptstyle H}(\Theta \cup \Gamma, \overline{\Theta^{\mathsf{c}} \setminus \Gamma}) \) respectively, obtained by restricting to those member sets that contain the vertex $a$. Observe that the bijection defined by \( f(B) = B \cup \Gamma \) does not add or remove any \( a \in V(H)\setminus\Gamma \); that is, it does not affect membership of \( a \). Hence, \( f \) restricts to a bijection from \( \mathcal{B}_{\scriptscriptstyle H \setminus \Gamma}(\{a\}\cup\Theta, \overline{\Theta^{\mathsf{c}} \cup N_{\scriptscriptstyle H}(\Gamma)}) \) to \( \mathcal{B}_{\scriptscriptstyle H}(\{a\}\cup\Theta \cup \Gamma , \overline{\Theta^{\mathsf{c}} \setminus \Gamma}) \). This justifies Equation \ref{eq2}.
\end{proof}

The Figure \ref{fig:p1in} gives an illustration of Lemma \ref{lem4}, showing the partition of \([n]\) into \(\Theta\) and \(\Theta^{\mathsf{c}}\), with the subset \(\Gamma \subseteq \Theta^{\mathsf{c}}\). Only the 2-layered vertices (in blue) are drawn, while the remaining vertices and edges are omitted. The graph \(H\) is shown in green, and the partition \(H \setminus \Gamma\) is indicated in purple.

\begin{lemma}\label{lem6}
     Let $H$ be a graph, and suppose $[n] \subseteq V(H)$ contains no isolated vertices. For any $\Theta\subseteq[n]$, $\Gamma_{\scriptscriptstyle 1},\Gamma_{\scriptscriptstyle 2}\subseteq\Theta^{\mathsf{c}}$ and $a\in V(H)$, if $\Gamma_{\scriptscriptstyle 1}\neq\Gamma_{\scriptscriptstyle 2}$, then,
    \begin{equation}\label{eq3}
       \mathcal{B}_{\scriptscriptstyle H\setminus \Gamma_{\scriptscriptstyle 1}}(\Theta,\overline{\Theta^{\mathsf{c}}\cup N_{\scriptscriptstyle H}( \Gamma_{\scriptscriptstyle 1}}))  
     \cap    \mathcal{B}_{\scriptscriptstyle H\setminus \Gamma_{\scriptscriptstyle 2}}(\Theta,\overline{\Theta^{\mathsf{c}}\cup N_{\scriptscriptstyle H}( \Gamma_{\scriptscriptstyle 2}}))=\emptyset
    \end{equation}
    \begin{equation}\label{eq4}
        \mathcal{B}_{\scriptscriptstyle H\setminus \Gamma_{\scriptscriptstyle 1}}(\{a\}\cup\Theta,\overline{\Theta^{\mathsf{c}}\cup N_{\scriptscriptstyle H}( \Gamma_{\scriptscriptstyle 1}}))  
     \cap    \mathcal{B}_{\scriptscriptstyle H\setminus \Gamma_{\scriptscriptstyle 2}}(\{a\}\cup\Theta,\overline{\Theta^{\mathsf{c}}\cup N_{\scriptscriptstyle H}( \Gamma_{\scriptscriptstyle 2}}))=\emptyset
    \end{equation}
\end{lemma}

\begin{proof}
Consider two distinct subsets \(\Gamma_{\scriptscriptstyle 1}\) and \(\Gamma_{\scriptscriptstyle 2}\) of \(\Theta^{\mathsf{c}}\). Without loss of generality, assume that there exists an element \(m \in \Gamma_{\scriptscriptstyle 1}\) such that \(m \notin \Gamma_{\scriptscriptstyle 2}\). Suppose \(B\) is an element of both \(\mathcal{B}_{\scriptscriptstyle H \setminus \Gamma_{\scriptscriptstyle 1}}(\Theta, \overline{\Theta^{\mathsf{c}} \cup N_{\scriptscriptstyle H}(\Gamma_{\scriptscriptstyle 1})})\) and \(\mathcal{B}_{\scriptscriptstyle H \setminus \Gamma_{\scriptscriptstyle 2}}(\Theta, \overline{\Theta^{\mathsf{c}} \cup N_{\scriptscriptstyle H}(\Gamma_{\scriptscriptstyle 2})})\). Since \(m \in \Gamma_{\scriptscriptstyle 1}\), the maximal stable sets in \(\mathcal{B}_{\scriptscriptstyle H \setminus \Gamma_{\scriptscriptstyle 1}}(\Theta, \overline{\Theta^{\mathsf{c}} \cup N_{\scriptscriptstyle H}(\Gamma_{\scriptscriptstyle 1})})\) avoid all vertices in \(N_{\scriptscriptstyle H}(m)\). Therefore, \(B \cap N_{\scriptscriptstyle H}(m) = \emptyset\). On the other hand, since \(B\) is also an element of \(\mathcal{B}_{\scriptscriptstyle H \setminus \Gamma_{\scriptscriptstyle 2}}(\Theta, \overline{\Theta^{\mathsf{c}} \cup N_{\scriptscriptstyle H}(\Gamma_{\scriptscriptstyle 2})})\), it must be a maximal stable set in \(H \setminus \Gamma_{\scriptscriptstyle 2}\) that excludes vertices from \(\Theta^{\mathsf{c}}\). Given that \(m \in \Theta^{\mathsf{c}}\) and \(m\) is not an isolated vertex, the maximality of \(B\) implies that \(B \cap N_{\scriptscriptstyle H}(m) \neq \emptyset\).

This contradiction shows that \(\mathcal{B}_{\scriptscriptstyle H \setminus \Gamma_{\scriptscriptstyle 1}}(\Theta, \overline{\Theta^{\mathsf{c}} \cup N_{\scriptscriptstyle H}(\Gamma_{\scriptscriptstyle 1})}) \cap \mathcal{B}_{\scriptscriptstyle H \setminus \Gamma_{\scriptscriptstyle 2}}(\Theta, \overline{\Theta^{\mathsf{c}} \cup N_{\scriptscriptstyle H}(\Gamma_{\scriptscriptstyle 2})})\) must be empty. Furthermore, it follows that \(\mathcal{B}_{\scriptscriptstyle H \setminus \Gamma}(\{a\} \cup \Theta, \overline{\Theta^{\mathsf{c}} \cup N_{\scriptscriptstyle H}(\Gamma)}) \subseteq \mathcal{B}_{\scriptscriptstyle H \setminus \Gamma}(\Theta, \overline{\Theta^{\mathsf{c}} \cup N_{\scriptscriptstyle H}(\Gamma)})\). Consequently, we deduce that \(\mathcal{B}_{\scriptscriptstyle H \setminus \Gamma_{\scriptscriptstyle 1}}(\{a\} \cup \Theta, \overline{\Theta^{\mathsf{c}} \cup N_{\scriptscriptstyle H}(\Gamma_{\scriptscriptstyle 1})}) \cap \mathcal{B}_{\scriptscriptstyle H \setminus \Gamma_{\scriptscriptstyle 2}}(\{a\} \cup \Theta, \overline{\Theta^{\mathsf{c}} \cup N_{\scriptscriptstyle H}(\Gamma_{\scriptscriptstyle 2})}) = \emptyset\).
\end{proof}

\begin{lemma}\label{lem7}
 Let $G$ be a graph with a decomposition $\{H,I\}$ such that $V(H)\cap V(I) = [n]$, where $[n]$ is a stable set and no vertex of $[n]$ is isolated in either $H$ or $I$. For $a \in V(I )\setminus [n]$ and $b \in [n]$, we have,
 \begin{equation}\label{eq5}
      w_{\scriptscriptstyle G}=\sum_{\scriptscriptstyle\Theta\subseteq[n]}\sum_{\scriptscriptstyle\Gamma\subseteq\Theta^{\mathsf{c}}}
       \left\{ w_{\scriptscriptstyle I\setminus\Gamma}(\Theta,\overline{\Theta^{\mathsf{c}}\cup N_{\scriptscriptstyle I}(\Gamma)})
\sum_{\scriptscriptstyle\Psi\subseteq\Theta^{\mathsf{c}}\setminus\Gamma}
 w_{\scriptscriptstyle H\setminus\Psi}(\Theta,\overline{\Theta^{\mathsf{c}}\cup N_{\scriptscriptstyle H}(\Psi)})\right\}
    \end{equation}
\begin{equation}\label{eq6}
w_{\scriptscriptstyle G}(a)=\sum_{\scriptscriptstyle\Theta\subseteq[n]}\sum_{\scriptscriptstyle\Gamma\subseteq\Theta^{\mathsf{c}}}
       \left\{ w_{\scriptscriptstyle I\setminus\Gamma}(\{a\}\cup\Theta,\overline{\Theta^{\mathsf{c}}\cup N_{\scriptscriptstyle I}(\Gamma)})
\sum_{\scriptscriptstyle\Psi\subseteq\Theta^{\mathsf{c}}\setminus\Gamma}
 w_{\scriptscriptstyle H\setminus\Psi}(\Theta,\overline{\Theta^{\mathsf{c}}\cup N_{\scriptscriptstyle H}(\Psi)})\right\}    \end{equation}
\begin{equation}\label{eq7}
w_{\scriptscriptstyle G}(b)=\sum_{\scriptscriptstyle\Theta\subseteq[n]}\sum_{\scriptscriptstyle\Gamma\subseteq\Theta^{\mathsf{c}}}
       \left\{ w_{\scriptscriptstyle I\setminus\Gamma}(\{b\}\cup\Theta,\overline{\Theta^{\mathsf{c}}\cup N_{\scriptscriptstyle I}(\Gamma)})
\sum_{\scriptscriptstyle\Psi\subseteq\Theta^{\mathsf{c}}\setminus\Gamma}
 w_{\scriptscriptstyle H\setminus\Psi}(\{b\}\cup\Theta,\overline{\Theta^{\mathsf{c}}\cup N_{\scriptscriptstyle H}(\Psi)})\right\}
    \end{equation}
   
\end{lemma}

\begin{proof}
    We will prove the equation \ref{eq5} by showing that,
    \[\mathcal{B}_{\scriptscriptstyle G} =\left\{B_{1}\cup B_{2}  \ \middle\vert \begin{array}{l}
     B_{1}\in \mathcal{B}_{\scriptscriptstyle I\setminus\Gamma}(\Theta,\overline{\Theta^{\mathsf{c}}\cup N_{\scriptscriptstyle I}(\Gamma)}),  B_{2}\in \mathcal{B}_{\scriptscriptstyle H\setminus\Psi}(\Theta,\overline{\Theta^{\mathsf{c}}\cup N_{\scriptscriptstyle H}(\Psi)}) \\
    \Theta\subseteq [n],\hspace{.5cm}\Gamma\subseteq\Theta^{\mathsf{c}},\hspace{.5cm}\Psi\subseteq\Theta^{\mathsf{c}}\setminus\Gamma
  \end{array}\right\}.
    \]

For simplicity, let us denote the family on the right-hand side by \(\mathcal{D}\). Consider any member set \(B_{1} \cup B_{2}\) in \(\mathcal{D}\). It is straightforward to observe that \(B_{1} \cup B_{2}\) forms a stable set in \(G\). If this set is not maximal in \(G\), then there must exist a vertex \(v \in V(G)\) such that \(\{v\} \cup B_{1} \cup B_{2}\) is also a stable set. However, due to the maximality of \(B_{1}\) and \(B_{2}\) in the subgraphs \(I \setminus \Gamma\) and \(H \setminus \Psi\), respectively, the vertex \(v\) cannot belong to either of these subgraphs. Since \(I \setminus \Gamma \cup H \setminus \Psi = G\), there is no such vertex \(v\) in the graph \(G\). Therefore, \(B_{1} \cup B_{2}\) must be a maximal stable set in \(G\), implying that \(B_{1} \cup B_{2} \in \mathcal{B}_{\scriptscriptstyle G}\) and hence \(\mathcal{D} \subseteq \mathcal{B}_{\scriptscriptstyle G}\).

Now, consider any \( B \in \mathcal{B}_{\scriptscriptstyle G} \). Define the subsets \(\Theta \subseteq [n]\) and \(\Theta^{\mathsf{c}} = [n] \setminus \Theta\) by setting \(\Theta = B \cap [n]\). Next, define \(\Gamma, \Psi \subseteq \Theta^{\mathsf{c}}\) as the sets of vertices in \([n]\) such that \(N_{\scriptscriptstyle I}(\Gamma) \cap B = \emptyset\) and \(N_{\scriptscriptstyle H}(\Psi) \cap B = \emptyset\), respectively. It is evident that \(\Gamma \cap \Psi = \emptyset\), because if there were a vertex \( v \in \Theta^{\mathsf{c}} \) with \( B \cap N_{\scriptscriptstyle G}(v) = B\cap (N_{\scriptscriptstyle I}(v) \cup N_{\scriptscriptstyle H}(v)) =  (B\cap N_{\scriptscriptstyle I}(v)) \cup (B \cap N_{\scriptscriptstyle H}(v))  = \emptyset \), this would contradict the maximality of \( B \) as a stable set\footnote{Since \( v \) and none of its neighbors would be included in \( B \).}. Furthermore, \( B \cap (I \setminus \Gamma) \) contains all vertices in \(\Theta\) but excludes vertices from both \(\Theta^{\mathsf{c}}\) and \(N_{\scriptscriptstyle I}(\Gamma)\). Notably, this set is stable since it is a subset of \( B \).

Assume that \( B \) is not maximal in \( I \setminus \Gamma \), and there exists a vertex \( v \in V(I \setminus \Gamma) \) such that \( v \notin B \) and \( (\{v\} \cup B) \cap (I \setminus \Gamma) \) is a stable set. For each \( u \in \Theta^{\mathsf{c}} \setminus \Gamma \), by definition, \( N_{\scriptscriptstyle I}(u) \cap B \neq \emptyset \), which implies that \( v \) cannot belong to \( \Theta^{\mathsf{c}} \setminus \Gamma \). Consequently, \( v \notin [n] \). Note that \( v \) also does not belong to \( N_{\scriptscriptstyle I}(\Theta) \), as \( (\{v\} \cup B) \cap (I \setminus \Gamma) \) remains stable. From this, one can observe that \( (\{v\} \cup B) \cap (H \setminus \Psi) \) is also stable, and consequently, \( B \cup \{v\} = ((\{v\} \cup B) \cap (I \setminus \Gamma)) \cup ((\{v\} \cup B) \cap (H \setminus \Psi)) \) forms a stable set in \( G \), which contradicts the maximality of \( B \). Therefore, \( B \cap (I \setminus \Gamma) \in \mathcal{B}_{\scriptscriptstyle I\setminus \Gamma}(\Theta, \overline{\Theta^{\mathsf{c}} \cup N_{\scriptscriptstyle I}(\Gamma)}) \). Similarly, one can show that \( B \cap (H \setminus \Psi) \in \mathcal{B}_{\scriptscriptstyle H\setminus\Psi}(\Theta, \overline{\Theta^{\mathsf{c}} \cup N_{\scriptscriptstyle H}(\Psi)}) \), and hence \( B \in \mathcal{D} \). This implies that \( \mathcal{B}_{\scriptscriptstyle G} \subseteq \mathcal{D} \).

With the application of Lemma \ref{lem6}, one can see that any \( B \in  \mathcal{D} \) has a unique \( B_{1} \) and \( B_{2} \) associated where \( B_{1} \cup B_{2} = B \), with \( B_{1}\in \mathcal{B}_{\scriptscriptstyle I\setminus\Gamma}(\Theta,\overline{\Theta^{\mathsf{c}}\cup N_{\scriptscriptstyle I}(\Gamma)}) \) and \( B_{2}\in \mathcal{B}_{\scriptscriptstyle H\setminus\Psi}
(\Theta,\overline{\Theta^{\mathsf{c}}\cup N_{\scriptscriptstyle H}(\Psi)}) \). This, together with \( \mathcal{B}_{\scriptscriptstyle G} =  \mathcal{D} \), derives Equation \ref{eq5}. Now notice that for any \( a \in V(I )\setminus [n] \), the family \( \mathcal{B}_{\scriptscriptstyle G}(a) \) consists of the maximal stable sets of \( G \) containing \( a \). Due to the one-to-one correspondence with \( \mathcal{D} \), this must be the collection of stable sets in \( \mathcal{D} \) that contain \( a \). That is,
\[
\mathcal{B}_{\scriptscriptstyle G}(a)=\left\{B_{1}\cup B_{2}  \ \middle\vert \begin{array}{l}
 B_{1}\in \mathcal{B}_{\scriptscriptstyle I\setminus\Gamma}(\{a\}\cup\Theta,\overline{\Theta^{\mathsf{c}}\cup N_{\scriptscriptstyle I}(\Gamma)}),  B_{2}\in \mathcal{B}_{\scriptscriptstyle H\setminus\Psi}(\Theta,\overline{\Theta^{\mathsf{c}}\cup N_{\scriptscriptstyle H}(\Psi)}) \\
 \Theta\subseteq [n],\hspace{.5cm}\Gamma\subseteq\Theta^{\mathsf{c}},\hspace{.5cm}\Psi\subseteq\Theta^{\mathsf{c}}\setminus\Gamma
\end{array}\right\}.
\]
Thus, Equation \ref{eq6} is established. Now, for Equation \ref{eq7}, consider any \( b \in [n] \). One can see that \( \mathcal{B}_{\scriptscriptstyle G}(b) \) is nothing but the collection of member sets of \( \mathcal{D} \) which contain \( b \). As \( b \in V(H) \cap V(I) \), it directly follows that
\[
\mathcal{B}_{\scriptscriptstyle G}(b)=\left\{B_{1}\cup B_{2}  \ \middle\vert \begin{array}{l}
 B_{1}\in \mathcal{B}_{\scriptscriptstyle I\setminus\Gamma}(\{b\}\cup\Theta,\overline{\Theta^{\mathsf{c}}\cup N_{\scriptscriptstyle I}(\Gamma)}),  B_{2}\in \mathcal{B}_{\scriptscriptstyle H\setminus\Psi}(\{b\}\cup\Theta,\overline{\Theta^{\mathsf{c}}\cup N_{\scriptscriptstyle H}(\Psi)}) \\
 \Theta\subseteq [n],\hspace{.5cm}\Gamma\subseteq\Theta^{\mathsf{c}},\hspace{.5cm}\Psi\subseteq\Theta^{\mathsf{c}}\setminus\Gamma
\end{array}\right\}.
\]
\end{proof}

An illustration of Lemma~\ref{lem7} is provided in Figure~\ref{fig:p2}. In this figure, the graph \( G \) is decomposed into two subgraphs \( H \) and \( I \), and the set of common vertices \([n] = [4]\) is partitioned into \(\Theta = \{1\}\) and \(\Theta^{\mathsf{c}} = \{2,3,4\}\). Furthermore, \(\Theta^{\mathsf{c}}\) is divided into the disjoint subsets \(\Gamma = \{2\}\) and \(\Psi = \{4\}\), as defined in the lemma.
Consider the maximal stable set \(\{1,6,9,11,13,14,17,18\} \in \mathcal{B}_{\scriptscriptstyle G}(\Theta,\overline{\Theta^{\mathsf{c}}})\). Observe that this set can be expressed as the union of \(\{1,6,9,17,18\}\) and \(\{1,11,13,14\}\). These two sets belong to the families \(\mathcal{B}_{\scriptscriptstyle H\setminus\Psi}(\Theta,\overline{\Theta^{\mathsf{c}}\cup N_{\scriptscriptstyle H}(\Psi)})\) and \(\mathcal{B}_{\scriptscriptstyle I\setminus\Gamma}(\Theta,\overline{\Theta^{\mathsf{c}}\cup N_{\scriptscriptstyle I}(\Gamma)})\), respectively. This demonstrates the bijection described in the proof of Lemma~\ref{lem7}.

\begin{proof}[of Theorem \ref{thm4}]
 Let $m \in  V(H) \cap V(I)$, and suppose $N_{\scriptscriptstyle I}(m) = \emptyset$. In this case, we can simply remove the vertex from $I$ to get a new decomposition. Similarly, if $N_{\scriptscriptstyle H}(m) = \emptyset$, we can remove $m$ from $H$. This allows us to redefine $H$ and $I$ in such a way that every vertex in $ V(H) \cap V(I)$ has neighbors in both $H$ and $I$. Let $[n]$ be the vertices common to these subgraphs $H$ and $I$. We will determine the number of maximal stable sets using Lemma \ref{lem7}. First, let's apply Lemma \ref{lem4} to Equation \ref{eq5}.
\begin{align}
    w_{\scriptscriptstyle G}
     & = \sum_{\scriptscriptstyle\Theta\subseteq[n]}\sum_{\scriptscriptstyle\Gamma\subseteq\Theta^{\mathsf{c}}}
       \left\{ w_{\scriptscriptstyle I\setminus\Gamma}(\Theta,\overline{\Theta^{\mathsf{c}}\cup N_{\scriptscriptstyle I}(\Gamma)})
\sum_{\scriptscriptstyle\Psi\subseteq\Theta^{\mathsf{c}}\setminus\Gamma}
w_{\scriptscriptstyle H\setminus \Psi}(\Theta,\overline{\Theta^{\mathsf{c}}\cup N_{\scriptscriptstyle H}(\Psi)})\right\}   \label{eqq9}
 \\
       & =  \sum_{\scriptscriptstyle\Theta\subseteq[n]}\sum_{\scriptscriptstyle\Gamma\subseteq\Theta^{\mathsf{c}}}
       \left\{ w_{\scriptscriptstyle I}(\Theta\cup\Gamma,\overline{\Theta^{\mathsf{c}}\setminus\Gamma})
\sum_{\scriptscriptstyle\Psi\subseteq\Theta^{\mathsf{c}}\setminus\Gamma}
w_{\scriptscriptstyle H}(\Theta\cup\Psi,\overline{\Theta^{\mathsf{c}}\setminus\Psi})\right\} \label{eqq10}
\end{align} 
Notice that,
\begin{align}
\sum_{\scriptscriptstyle\Psi\subseteq\Theta^{\mathsf{c}}\setminus\Gamma}
w_{\scriptscriptstyle H}(\Theta\cup\Psi,\overline{\Theta^{\mathsf{c}}\setminus\Psi}) & = \sum_{\scriptscriptstyle\Psi\subseteq\Theta^{\mathsf{c}}\setminus\Gamma}
w_{\scriptscriptstyle H}(\Theta\cup\Psi,\overline{\Gamma\cup ((\Theta^{\mathsf{c}}\setminus\Gamma)\setminus\Psi)})   \label{eqq11}\\
        & = w_{\scriptscriptstyle H}(\Theta,\overline{\Gamma})   \label{eqq12}
\end{align}
 This identity holds because we are summing over all possible subsets \(\Psi \subseteq \Theta^{\mathsf{c}} \setminus \Gamma\). Next, we introduce a new variable \(\Omega = \Theta \cup \Gamma\), with \(\Omega^{\mathsf{c}} = [n] \setminus \Omega = \Theta^{\mathsf{c}} \setminus \Gamma\). This allows us to express $w_{\scriptscriptstyle I}(\Theta \cup \Gamma, \overline{\Theta^{\mathsf{c}} \setminus \Gamma})$ as $w_{\scriptscriptstyle I}(\Omega, \overline{\Omega^{\mathsf{c}}})$. We can now apply this change of variables after substituting \eqref{eqq12} into equation \eqref{eqq10}.
\begin{align}
    w_{\scriptscriptstyle G}
     & = \sum_{\scriptscriptstyle\Theta\subseteq[n]}\sum_{\scriptscriptstyle\Gamma\subseteq\Theta^{\mathsf{c}}}
       \left\{ w_{\scriptscriptstyle I}(\Theta\cup\Gamma,\overline{\Theta^{\mathsf{c}}\setminus\Gamma})
w_{\scriptscriptstyle H}(\Theta,\overline{\Gamma})\right\}   \tag{[Substituting equation \ref{eqq12}]}
 \\
       & =  \sum_{\scriptscriptstyle\Omega\subseteq[n]}\sum_{\scriptscriptstyle\Theta\subseteq\Omega}
       \left\{ w_{\scriptscriptstyle I}(\Omega,\overline{\Omega^{\mathsf{c}}})
w_{\scriptscriptstyle H}(\Theta,\overline{\Omega\setminus\Theta})\right\}  \tag{[Change of variables]}
\\
       & =  \sum_{\scriptscriptstyle\Omega\subseteq[n]}
       \left\{ w_{\scriptscriptstyle I}(\Omega,\overline{\Omega^{\mathsf{c}}})  \sum_{\scriptscriptstyle\Theta\subseteq\Omega}
w_{\scriptscriptstyle H}(\Theta,\overline{\Omega\setminus\Theta})\right\} \nonumber \\
       & =  \sum_{\scriptscriptstyle\Omega\subseteq[n]}
       \left\{ w_{\scriptscriptstyle I}(\Omega,\overline{\Omega^{\mathsf{c}}}) w_{\scriptscriptstyle H}\right\}  \nonumber
       \\
       & =  w_{\scriptscriptstyle I} w_{\scriptscriptstyle H} \label{eqq13}
\end{align}
 \end{proof}

\begin{lemma}\label{newl}
Let \(G\) be a graph with a decomposition \(\{H,I\}\) such that \(V(H)\cap V(I) = [n]\), where \([n]\) is a stable set. Assume that no vertex of \([n]\) is isolated in either \(H\) or \(I\), and that every vertex of \([n]\) is 2-layered. Let \(a \in V(I)\setminus [n]\) and \(b \in [n]\). Then,
\begin{equation}\label{eql9}
w_{\scriptscriptstyle G}(a) = w_{\scriptscriptstyle I}(a)w_{\scriptscriptstyle H}
\end{equation}
\begin{equation}\label{eql10}
w_{\scriptscriptstyle G}(b) \leq w_{\scriptscriptstyle I}(b)w_{\scriptscriptstyle H}(b)
\end{equation}
\end{lemma}

\begin{proof}
  Very similar to the proof of Theorem \ref{thm4}, applying Lemma \ref{lem4} to Equation \ref{eq6}, it can be deduced that 
\( w_{\scriptscriptstyle G}(a) = w_{\scriptscriptstyle I}(a)\, w_{\scriptscriptstyle H} \) 
for any \( a \in V(I)\setminus [n] \). In order to prove Equation \ref{eql10}, we apply Lemma \ref{lem4} to Equation \ref{eq7}. Now consider a vertex \(b\in [n]\). One should note that Lemma \ref{lem4} states 
\( w_{\scriptscriptstyle I \setminus \Gamma}(\{b\}\cup \Theta, \overline{\Theta^{\mathsf{c}}\cup N_{\scriptscriptstyle I}(\Gamma)}) 
= w_{\scriptscriptstyle I}(\{b\}\cup \Theta \cup \Gamma, \overline{\Theta^{\mathsf{c}} \setminus \Gamma}) \) 
whenever \( b \notin \Gamma \). Observe that for every \( \Gamma \) with \( b \in \Gamma \), we have 
\( w_{\scriptscriptstyle I \setminus \Gamma}(\{b\} \cup \Theta, \overline{\Theta^{\mathsf{c}} \cup N_{\scriptscriptstyle I}(\Gamma)}) = 0 \). Therefore, some terms vanish, and since the steps are otherwise identical to the proof of Theorem \ref{thm4}, we obtain 
\( w_{\scriptscriptstyle G}(b) \leq w_{\scriptscriptstyle I}(b)\, w_{\scriptscriptstyle H}(b) \).
\end{proof}

\begin{proof}[of Theorem \ref{thm42}]
 We redefine \(H\) and \(I\) so that no common vertex is isolated in either subgraph, as in the proof of Theorem \ref{thm4}. Let \([n]\) denote the set of common vertices. Since all vertices of \([n]\) belong to the same bipartite class, \([n]\) forms a stable set. Moreover, as each vertex of \([n]\) is 2-layered, we may apply Theorem \ref{thm4} and Lemma \ref{newl}. Consider \( r \in V(I) \) a rare vertex in \( I \). If \( r \in V(I) \setminus [n] \), we have $w_{\scriptscriptstyle G}-2w_{\scriptscriptstyle G}(r)=w_{\scriptscriptstyle I}w_{\scriptscriptstyle H}-2w_{\scriptscriptstyle I}(r)w_{\scriptscriptstyle H}=(w_{\scriptscriptstyle I}-2w_{\scriptscriptstyle I}(r))w_{\scriptscriptstyle H}\geq 0$. Now, if \( r \in [n] \), then $w_{\scriptscriptstyle G}-2w_{\scriptscriptstyle G}(r)\geq w_{\scriptscriptstyle I}w_{\scriptscriptstyle H}-2w_{\scriptscriptstyle I}(r)w_{\scriptscriptstyle H}(r)\geq(w_{\scriptscriptstyle I}-2w_{\scriptscriptstyle I}(r))w_{\scriptscriptstyle H}\geq 0$. Therefore, every rare vertex in \( I \) is also rare in \( G \).
\end{proof}

\begin{figure}[h!]
\centering
\begin{minipage}{.475\textwidth}
  \centering

 \includegraphics[width=0.75\textwidth]{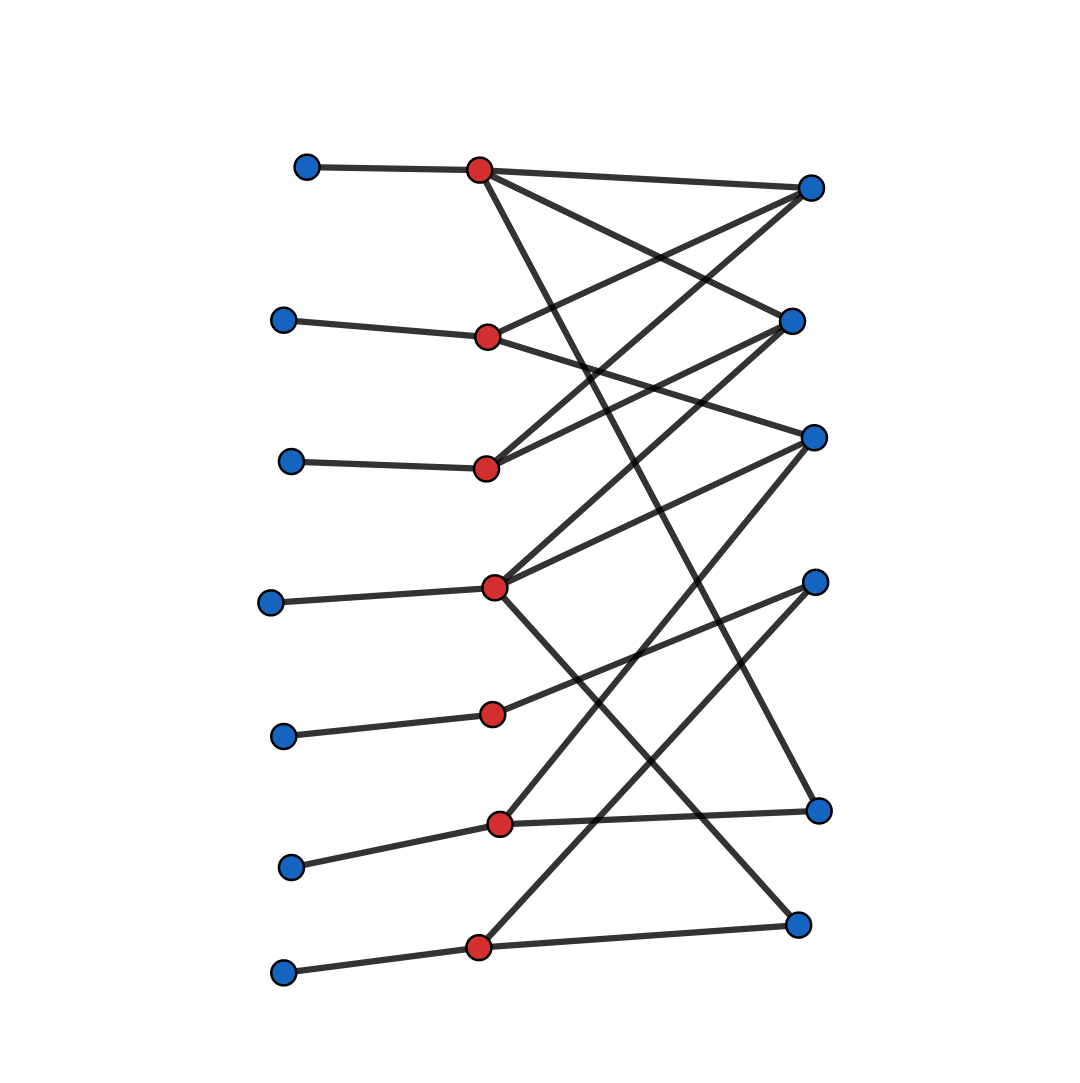}
 
 \caption{The bipartite graph shown consists of two bipartite classes represented by red and blue vertices. As all the vertices in red bipartite class has a pendant vertex adjacent to it, the graph satisfies the conjecture.}
  \label{fig:c}

\end{minipage}%
\hfill
\begin{minipage}{.475\textwidth}
  \centering
 
  \includegraphics[width=0.75\textwidth]{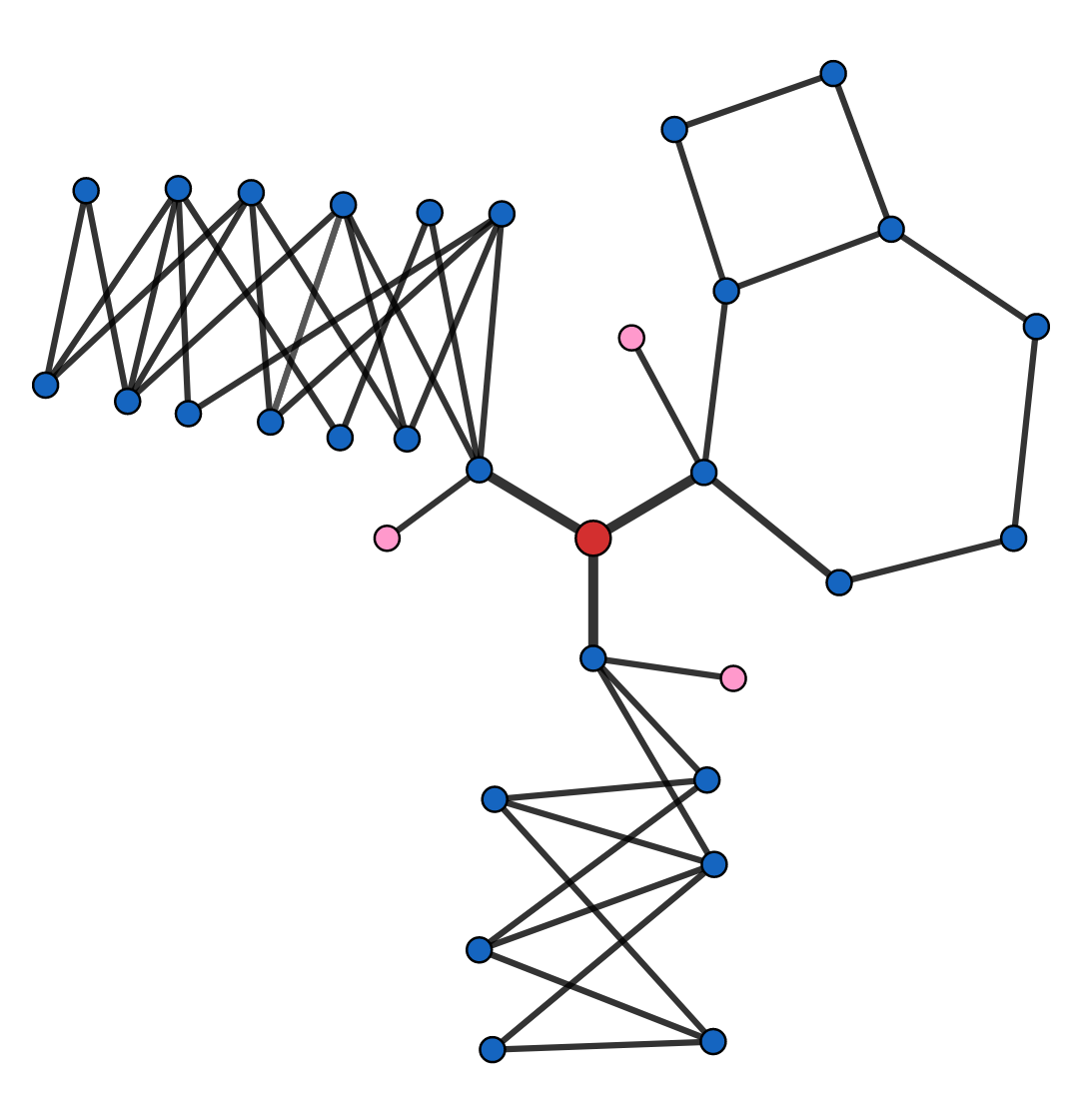}

\caption{Three merged graphs, with a 2-layered vertex highlighted in red. The resultant graph satisfies the conjecture because each of the three individual graphs satisfies the conjecture.}

  \label{fig:p}
 
\end{minipage}
\end{figure}

The theorem's implications extend to solving the conjecture for previously unexplored graph classes, opening up new avenues of research. Now, let's discuss some of the notable consequences stemming from this result. All results concerning the set version of the Union Closed Sets Conjecture presented below are directly derived from the corresponding graph results by applying Proposition \ref{prop31} within the respective graph results.

\begin{proposition}[Set Version of Theorem \ref{thm42}]
     Let \(\mathcal{F}_{\scriptscriptstyle 1}\) and \(\mathcal{F}_{\scriptscriptstyle 2}\) be two families of sets that do not share any common sets. Let \(A\) denote the set of elements that are common to both families, i.e., \(A = U(\mathcal{F}_{\scriptscriptstyle 1}) \cap U(\mathcal{F}_{\scriptscriptstyle 2})\). Assume that every \(S \in \mathcal{F}_{\scriptscriptstyle 1} \cup \mathcal{F}_{\scriptscriptstyle 2}\) with \(S \cap A \neq \emptyset\) contains an element of frequency 1. If \(\langle \mathcal{F}_{\scriptscriptstyle 1} \rangle\) or \(\langle \mathcal{F}_{\scriptscriptstyle 2} \rangle\) satisfies the UCC condition, then \(\langle \mathcal{F}_{\scriptscriptstyle 1} \cup \mathcal{F}_{\scriptscriptstyle 2} \rangle\) also satisfies the UCC condition.
\end{proposition}

\begin{proposition}\label{prop48}
   Let \( G \) be a bipartite graph and \( v \in V(G) \) be such that there exists at least one pendant vertex adjacent to \( v \) and all the non-pendant neighbors of \( v \) are 2-layered. Then \( G \) satisfies UCC condition. Furthermore, the vertex $v$ and all of its neighbors are rare in $G$.
\end{proposition}
\begin{proof}
    Define the set of pendant neighbors of \(v\) as 
\(P = \{u \in N_{\scriptscriptstyle G}(v) \mid d(u) = 1\},
\)where \(d(u)\) denotes the degree of \(u\). Let $I$ be the subgraph induced by vertex $v$, while $H$ is the subgraph induced by the vertex subset $V(G) \setminus (P \cup \{v\})$. Clearly, $V(H) \cap V(I) = N_{\scriptscriptstyle G}(v) \setminus P$, which corresponds to the non-pendant neighbors of $v$. 

     Observe that all vertices in $N(v) \setminus P$ are 2-layered vertices, and $\{I, H\}$ forms a decomposition of $G$. It is important to note that $I$ is a star graph, implying that all of its vertices are rare within $I$. Consequently, by applying Theorem \ref{thm42}, we can conclude that both vertex $v$ and its neighbors in $G$ are rare in $G$, thereby establishing that $G$ satisfies the conjecture.
\end{proof}

\begin{proposition}[Set version of Proposition \ref{prop48}]
   Let $\mathcal{F}$ be a family of sets. Consider a member set $A \in \mathcal{F}$ that contains at least one element with frequency one. Suppose that for every set $S \in \mathcal{F}$, where $S \cap A \neq \emptyset$, there exists an element within $S$ that has frequency one. Under these conditions, the family $\langle\mathcal{F}\rangle$ satisfies the UCC condition, with all elements of $A$ being abundant.
\end{proposition}

\begin{corollary}\label{cor49}
     Let $G$ be a bipartite graph with no isolated vertices, in which every vertex in one of its bipartite classes is adjacent to at least one pendant vertex. Then, $G$ satisfies UCC condition. In fact, all the vertices of $G$ are rare.
\end{corollary}

\begin{corollary}[Set version of Corollary \ref{cor49}]
     Let $\mathcal{F}$ be a family of sets where each of its member sets contains an element of frequency 1. Then $\langle\mathcal{F}\rangle$ satisfies the UCC condition. 
\end{corollary}

\begin{proposition}\label{prop411}
   Suppose we have a collection of bipartite graphs $G_{\scriptscriptstyle 1}, G_{\scriptscriptstyle 2}, \ldots, G_{\scriptscriptstyle k}$, where no two graphs share any common vertices. Additionally, let $v_{\scriptscriptstyle 1}, v_{\scriptscriptstyle 2}, \ldots, v_{\scriptscriptstyle k}$ denote the 2-layered vertices within graphs $G_{\scriptscriptstyle 1}, G_{\scriptscriptstyle 2}, \ldots, G_{\scriptscriptstyle k}$, respectively. By merging these $k$ vertices, we obtain a new graph $G$. Under this construction, any vertex that is rare in one of the original graphs will remain rare in the composite graph $G$.
\end{proposition}

\begin{proof}
   Let $v$ be the vertex in $G$ resulting from the merging of all $k$ vertices. It is evident that $v$ is a 2-layered vertex, allowing us to apply Theorem \ref{thm42}.
\end{proof}

As an illustration, Figure \ref{fig:c} provides an example of Corollary \ref{cor49}, while Figure \ref{fig:p} demonstrates an example of Proposition \ref{prop411}.

\begin{proposition}[Set version of Proposition \ref{prop411}]
   Let $\mathcal{F}_{\scriptscriptstyle 1}, \mathcal{F}_{\scriptscriptstyle 2}, \ldots, \mathcal{F}_{\scriptscriptstyle k}$ be a collection of families of sets such that their universes are mutually exclusive. From each family $\mathcal{F}_{\scriptscriptstyle i}$, choose an element $a_{\scriptscriptstyle i}$ such that every member set containing $a_{\scriptscriptstyle i}$ includes an element with frequency 1. Now, set $a = a_{\scriptscriptstyle 1} = a_{\scriptscriptstyle 2} = \ldots = a_{\scriptscriptstyle k}$ and define $\mathcal{F} = \bigcup_{i=1}^{k} \mathcal{F}_{\scriptscriptstyle i}$. Under this construction, all abundant elements in any of the $\langle \mathcal{F}_{\scriptscriptstyle i} \rangle$ remain abundant in $\langle \mathcal{F} \rangle$.
\end{proposition}

\section{Future Research}\label{secf}

Several graph classes, such as subcubic graphs, chordal bipartite graphs, and series-parallel graphs, are already known to satisfy the Union Closed Sets Conjecture. However, many families still remain unresolved; for example, the conjecture has not yet been established even for the classes of regular bipartite graphs, planar bipartite graphs, and bipartite graphs of treewidth three. This represents an important direction for further exploration.

In this paper, we identified some additional graph classes that satisfy the conjecture, but more significantly, we developed a proof technique that may inspire further work. Our method crucially requires that the common vertices in the decomposition be 2-layered. A natural next step is to investigate situations where some, or even all, of these vertices fail to be 2-layered. For instance, one may ask whether there exist particular cases in which the number of maximal stable sets equals the product of the numbers of maximal stable sets of the decompositions, despite the lack of the 2-layered property. Establishing such results could pave the way toward a more general theorem.  

Another promising perspective is to study the conjecture by systematically decomposing graphs into two or more subgraphs and analyzing the structure of maximal stable sets within each piece. Even in cases where the common vertices are not 2-layered, one can investigate whether useful bounds can still be obtained using ideas similar to those presented here. To the best of our knowledge, the approach of analyzing the Union Closed Sets Conjecture through graph decompositions has not been explored in the literature. Therefore, pursuing this direction may open new avenues for understanding the conjecture.

\acknowledgements
I would like to express my sincere gratitude to my adviser, Prof. Rogers Mathew, for his invaluable guidance and support.

\bibliographystyle{abbrvnat}
\bibliography{references}

\end{document}